\documentclass[letterpaper, 10 pt, conference]{ieeeconf}  

\IEEEoverridecommandlockouts                              
\title{\LARGE \bf
Towards Almost Global Synchronization on the Stiefel Manifold
}

\author{Johan Markdahl, Johan Thunberg, and Jorge Gon\c{c}alves
	\thanks{J. Markdahl, J. Thunberg, and J. Gon\c{c}alves are with the Luxembourg Centre for Systems Biomedicine (\textsc{lcsb}),
		University of Luxembourg. J.~Markdahl is the corresponding author:
		{\tt\small markdahl@kth.se}}%
}

\usepackage{subdepth} 

\usepackage{graphics} 
\usepackage{epsfig} 

\usepackage{amssymb} 
\usepackage{amsmath,psfrag,color,graphicx} 
\usepackage{graphicx}

\usepackage[small,hang,bf,up,labelsep=period]{caption} 
\usepackage{quoting} 
\usepackage{lettrine} 
\usepackage{tocloft} 
\usepackage{titletoc}
\usepackage{fmtcount}
\usepackage{calc} 
\usepackage{cite} 
\usepackage{multicol} 
\usepackage{centernot}
\usepackage{nicefrac}
\usepackage{dsfont}
\usepackage{booktabs}

\usepackage{lipsum}
\usepackage{psfrag}
\usepackage[export]{adjustbox} 

\PassOptionsToPackage{cmyk}{xcolor}
\usepackage[cmyk]{xcolor}
\selectcolormodel{cmyk}

\usepackage{textcomp}

\usepackage{txfonts}
\usepackage{lmodern}
\usepackage{times}

\usepackage{scalerel}[2016/12/29]
\DeclareFontEncoding{LS1}{}{}
\DeclareFontSubstitution{LS1}{stix}{m}{n}
\newcommand*\laplac{%
	\text{%
		\fontencoding{LS1}%
		\fontfamily{stixfrak}%
		\fontseries{\textmathversion}%
		\fontshape{n}%
		\selectfont\symbol{"CF}}}
\makeatletter
\newcommand*\textmathversion{\csname textmv@\math@version\endcsname}
\newcommand*\textmv@normal{m}
\newcommand*\textmv@bold{b}
\makeatother
\DeclareMathOperator{\inabla}{\raisebox{1.1pt}{\scaleobj{0.85}{\laplac}}}

\newcommand{\ve}[2][]{\ensuremath{\boldsymbol{\mathrm{#2}}}_{#1}}
\newcommand{\vet}[2][]{\ensuremath{\smash{\boldsymbol{\mathrm{#2}}^{\top}_{#1}}}}
\newcommand{\vd}[2][]{\ensuremath{\dot{\boldsymbol{\mathrm{#2}}}_{#1}}}

\newcommand{\ma}[2][]{\ensuremath{\boldsymbol{\mathrm{#2}}}_{#1}}
\newcommand{\mat}[2][]{\ensuremath{\boldsymbol{\mathrm{#2}}^{\!\top}_{#1}}}
\newcommand{\md}[2][]{\ensuremath{\dot{\boldsymbol{\mathrm{#2}}}}_{#1}}

\DeclareMathOperator{\vect}{\ensuremath{\mathrm{vec}}}

\DeclareMathOperator{\skews}{\ensuremath{\mathrm{skew}}}

\DeclareMathOperator{\syms}{\ensuremath{\mathrm{sym}}}

\newcommand{\R}{\ensuremath{\mathds{R}}}

\newcommand{\N}{\ensuremath{\mathds{N}}}

\newcommand{\SOT}{\ensuremath{\mathsf{SO}(3)}}

\newcommand{\SO}{\ensuremath{\mathsf{SO}(n)}}

\newcommand{\so}{\ensuremath{\mathsf{so}(n)}}

\newcommand{\St}{\ensuremath{\mathsf{St}(p,n)}}

\newcommand{\Sn}{\ensuremath{\mathsf{S}^{n}}}

\newcommand{\GC}{\ensuremath{\mathsf{S}^1}}

\newcommand{\ie}{\textit{i.e.}, }

\newcommand{\eg}{\textit{e.g.}, }

\newcommand{\mtr}{\hspace{-0.3mm}\ensuremath{^\top}}

\newcommand\raiseT[2]{\setbox0\hbox{$#1{#2}$}\raise\dp0\box0}

\newcommand{\V}{\ensuremath{\mathcal{V}}}

\newcommand{\E}{\ensuremath{\mathcal{E}}}

\newcommand{\etc}{\emph{etc.\ }}

\DeclareMathOperator{\trace}{\ensuremath{\mathrm{tr}}}

\DeclareMathOperator{\im}{\ensuremath{\mathrm{Im}}}

\newcommand{\Ni}{\ensuremath{\mathcal{N}_i}}

\makeatletter
\newcommand{\raisemath}[1]{\mathpalette{\raisem@th{#1}}}
\newcommand{\raisem@th}[3]{\raisebox{#1}{$#2#3$}}
\makeatother

\newcommand{\ts}[2][]{\ensuremath{\mathsf{T}_{#2}#1}}

\definecolor{kthbluergb}{RGB}{25,84,166}
\definecolor{kthbluecmyk}{cmyk}{1,0.55,0,0}

\definecolor{kthblueA}{RGB}{25,84,166}
\definecolor{kthblueB}{RGB}{46,124,192}
\definecolor{kthblueC}{RGB}{112,153,209}
\definecolor{kthblueD}{RGB}{164,186,225}
\definecolor{kthblueE}{RGB}{211,220,241}

\newcounter{counter} 

\newtheorem{theorem}[counter]{Theorem}

\newtheorem{definition}[counter]{Definition}
\newtheorem{example}[counter]{Example}

\newtheorem{remark}[counter]{Remark}


\newcounter{parentnumber}

\allowdisplaybreaks

\begin{document}

\maketitle
\thispagestyle{empty}
\pagestyle{empty}

\begin{abstract}
A graph $\mathcal{G}$ is referred to as $\GC$-synchronizing if, roughly speaking, the Kuramoto-like model whose interaction topology is given by $\mathcal{G}$ synchronizes almost globally. The Kuramoto model evolves on the unit circle, \ie the $1$-sphere $\smash{\GC}$. This paper concerns generalizations of the Kuramoto-like model and the concept of synchronizing graphs on the Stiefel manifold $\St$. Previous work on state-space oscillators have largely been influenced by results and techniques that pertain to the $\smash{\GC}$-case. It has recently been shown that all connected graphs are $\smash{\Sn}$-synchronizing for all $n\geq2$. The previous point of departure may thus have been overly conservative. The $n$-sphere is a special case of the Stiefel manifold, namely $\mathsf{St}(1,n+1)$. As such, it is natural to ask for the extent to which the results on $\mathcal{S}^{n}$ can be extended to the Stiefel manifold. This paper shows that all connected graphs are $\St$-synchronizing provided the pair $(p,n)$ satisfies $p\leq \tfrac{2n}{3}-1$. 
\end{abstract}

\section{Introduction}

\noindent Scalability is a key advantage of distributed approaches to feedback control of multi-agent systems \cite{sarlette2009geometry}. It is achieved by feedback laws with linear computational complexity that can be executed when interactions between most agents is indirect, \eg when the communications topology is given by a path or cycle graph. An often overlooked aspect of scalability is control performance at a large distance from nominal operating conditions. Consider the case of $N$ homogeneous agents, the state of each belonging to a compact manifold $\mathcal{M}$. Many results in the literature concerns the case when  all agents belong to some convex subset of $\mathcal{M}$ \cite{afsari2011riemannian,tron2013riemannian,hartley2013rotation}. A typical result is the guaranteed convergence of a synchronization algorithm  \cite{zhu2013synchronization,thunberg2014distributed,lageman2016consensus}. If the initial states follow a uniform distribution on $\mathcal{M}$, then the probability that all agents are contained in such a convex set decreases exponentially with $N$. This is a case of poor scaling. Ideally, the probability of convergence should be $1$ independently of $N$. Such performance is achieved by almost globally convergent algorithms. This paper concerns the problem of establishing almost global convergence of a class of continuous time consensus protocols for multi-agent system that evolve on the compact, real Stiefel manifold $\St$. Each of the systems under consideration is an intrinsic gradient descent flow of a basic, quadratic potential function.


A graph $\mathcal{G}$ is $\GC$-synchronizing if, roughly speaking, the consensus manifold is an almost globally stable equilibrium manifold of a Kuramoto-like oscillator model with topology $\mathcal{G}$. Examples of such graphs include the complete graph, acyclic graphs, and sufficiently dense graphs \cite{dorfler2014synchronization}. The survey \cite{sepulchre2011consensus} ponders the question of which combinations of graphs, manifolds, and consensus protocols lead to almost global synchrony. A key result is that not all undirected graphs are $\GC$-synchronizing, but there are protocols that yield other, almost globally synchronizing closed-loop systems on $\GC$  \cite{scardovi2007synchronization,sarlette2009geometry}. A generalization to $\SOT$ has been established \cite{tron2012intrinsic}. Moreover, some  protocols also converge in the case of quasi-strongly connected digraphs on the Stiefel manifold \cite{thunberg2017dynamic}. The control strategy is based on the use of auxiliary variables that must be communicated between agents. As such,  it is intended for use in engineering systems rather than as a model of self-organizing systems of coupled oscillators that are observed in nature.



Consensus protocols for agents whose dynamics evolve over linear spaces are well understood \cite{mesbahi2010graph}. Moving forward, it is natural to look at homogeneous spaces \cite{sarlette2009consensus}. All points on a homogeneous space are similar, making them well-suited for hosting multi-agent systems where relative and aggregate quantities are the main focus. This paper studies a consensus protocol on a homogeneous space: a Kuramoto-like model on the compact, real Stiefel manifold. Special cases of $\St$ include the $(n-1)$-sphere when $p=1$, the special orthogonal group when $p=n-1$, and the orthogonal group when $p=n$. As such, these results are of relevance in a number of applications including cooperative reduced and full attitude control in the cases of $(p,n)=(1,3)$ \cite{olfati2006swarms,li2014unified} and $(p,n)=(2,3)$ \cite{sarlette2009autonomous} respectively. 
The Stiefel manifold for $p\in\{2,\ldots,n-2\}$ is typically not used to model physical systems, which is the main area of application for continuous time control systems (although application of continuous time dynamical systems theory to numerical linear algebra problems is not unheard of \cite{helmke2012optimization}). 

The authors of this paper have showed that all connected graphs are $\Sn$-synchronizing for all $n\in\N\backslash\{1\}$ \cite{markdahl2017tac}. This result is of interest since it is unexpected; it could not have been interpolated from the previous research concerning special cases such as $\GC$-synchronizing graphs for the Kuramoto model and more general findings concerning almost global synchrony on $\SOT$ \cite{tron2012intrinsic}. This paper generalizes those results to the Stiefel manifold where we show that all connected graphs are $\St$-synchronizing provided the pair $(p,n)$ satisfies $p\leq\tfrac{2n}{3}-1$. This inequality is imposed for technical reasons  and is unlikely to have any interesting interpretation. Previous research on almost global synchronization over homogeneous manifolds departs from the negative result that not all undirected graphs are $\GC$-synchronizing \cite{sarlette2009geometry,sarlette2009consensus}. We show that such pessimism may be unfounded since the circle is actually a pathological case, providing hope that results akin to those of this paper also apply to other Riemannian manifolds. 

\section{Problem Formulation}

\subsection{Preliminaries}

\noindent The following notation is used in this paper. 
The Euclidean inner product of $\ma{X}, \ma{Y}\in\R^{n\times p}$ is $g_e(\ma{X},\ma{Y})=\langle\ma{X},\ma{Y}\rangle=\trace\mat{X}\ma{Y}$. The norm of $\ma{X}$ is given by $\|\ma{X}\|=\smash{\langle\ma{X},\ma{X}\rangle^{\frac12}}$, \ie the Frobenius norm. The compact real Stiefel manifold is considered an embedded matrix manifold in $\R^{n\times p}$ \cite{absil2009optimization},
\begin{align*}
\St=\{\ma{S}\in\R^{n\times p}\,|\,\mat{S}\ma{S}=\ma{I}\}.
\end{align*}
The pair $(\St,g_e)$ forms a smooth Riemannian manifold. An important Stiefel manifold is $\mathsf{St}(1,n+1)$, \ie the $n$-sphere which we denote $\Sn$.

Define the projections $\skews:\R^{n\times n}\rightarrow\so:\ma{X}\mapsto\tfrac{1}{2}(\ma{X}-\mat{X})$ and $\syms:\R^{n\times n}\rightarrow\so^\perp:\ma{X}\mapsto\tfrac12(\ma{X}+\mat{X})$. The tangent space of $\St$ at $\ma{X}$ is given by
\begin{align*}
\ts[\St]{\ma{S}}&=\{\ma{\Delta}\in\R^{n\times p}\,|\,\syms\mat{S}\ma{\Delta}=\ma{0}\}.
\end{align*}
The projection on the tangent space, $\Pi:\R^{n\times p}\times\St\rightarrow\ts[\St]{\ma{S}}$, is given by 
\begin{align*}
\Pi(\ma{X},\ma{S})=\ma{S}\skews\mat{S}\ma{X}+(\ma[n]{I}-\ma{S}\mat{S})\ma{X}.
\end{align*}
From a computational perspective, it is sometimes preferable to use the equivalent expression
\begin{align*}
\Pi(\ma{X},\ma{S})=\ma{X}-\ma{S}\syms\mat{S}\ma{X}.
\end{align*}
%

The gradient on $\St$ (in terms of the Euclidean inner product $g_e$) of a function $U:\St\rightarrow\R$ is given by
\begin{align*}
\inabla U=\Pi\,\nabla V,
\end{align*}
where $V$ is any smooth extension of $U$ on $\R^{n\times p}$, and $\nabla$ denotes the gradient in Euclidean space.

Each element $i\in\V$ also denotes an agent. Items associated with a specific agent $i$ carry the subindex $i$; we let $\ma[i]{S}\in\St$ denote the state of an agent, $\Pi_i$ the projection onto the tagent space $\ma[i]{S}$, $\inabla_i U$ the gradient of $U$ with respect to $\ma[i]{S}$, \etc 

\subsection{Distributed Control Design}

\noindent The consensus submanifold $\mathcal{C}$ of an analytic Riemannian manifold $(\mathcal{M},g)$ is the set of equilibria
\begin{align*}
\mathcal{C}&=\{(x_i)_{i=1}^N\in\mathcal{M}^N\,|\,x_i=x_j,\,\forall\,\{i,j\}\in\E\}.
\end{align*}
The consensus set is a manifold, in fact it is diffeomorphic to $\mathcal{M}$ using the map $
(x_i)_{i=1}^N\mapsto x_1$.

Given a graph $(\V,\E)$, define the potential function $U:\mathcal{M}^N\rightarrow\R$ by
\begin{align*}
U=\sum_{e\in\E}f_{ij}(d_g(x_i,x_j)),
\end{align*}
where $f_{ij}:\R\rightarrow[0,\infty)$, $d_g$ is the geodesic distance on $\mathcal{M}$ in terms of $g$, and $e$ is on the form $\{i,j\}$. The consensus seeking system on $\mathcal{M}$ obtained from $U$ is the gradient descent flow
\begin{align}
(\dot{x}_i)_{i=1}^N=(-\inabla_i U)_{i=1}^N,\label{eq:descent}
\end{align}
where $x_i(0)\in\mathcal{M}$ for all $i\in\V$.

Agent $i$ does not have access to $U$, but can calculate
\begin{align*}
U_i=\tfrac12\sum_{j\in\Ni}f_{ij}(d_g(x_i,x_j))
\end{align*}
at its current position. Symmetry of $d_g$ gives $U=\sum_{i\in\V}U_i$ whereby it follows that $\inabla_i U_i=\inabla_i U$. From a control design perspective, we can assume that the dynamics of each agent take the form $\dot{x}_i=u_i$ with $u_i\in\ts[\mathcal{M}]{i}$. Since agent $i$ can evaluate $U_i$ at its current position, it is reasonable to assume that it can also calculate $u_i=-\inabla_i U_i$.

\subsection{$\St$-Synchronizing Graphs}

\begin{definition}
A graph $\mathcal{G}$ is $\mathcal{M}$-synchronizing if all minimizers of $U$ belong to $\mathcal{C}$.
\end{definition}

The property of being $\mathcal{M}$-synchronizing, which we have adopted from \cite{sarlette2009geometry}, does not explicitly reference the specific function $U$ under consideration. For the purpose of this paper, 
we limit consideration to potential functions $U$ on matrix manifolds $\mathcal{M}\subset\R^{n\times m}$ of the following form
\begin{align}
U=\tfrac{1}{2}\sum_{e\in\E}a_{ij}\|\ma[i]{X}-\ma[j]{X}\|^2,\label{eq:U}
\end{align}
where $\ma[i]{X}\in\R^{n\times m}$, and the constants $a_{ij}$ are strictly positive and symmetric, \ie  $a_{ji}=a_{ij}$. On the Stiefel manifold $\St\subset\R^{n\times p}$, this reduces to
\begin{align}\label{eq:VSt}
U=\sum_{e\in\E}a_{ij}(p-\langle\ma[i]{S},\ma[j]{S}\rangle),
\end{align}
since $\|\ma[i]{S}\|^2=p$ for all $i\in\V$.

\begin{definition}
An equilibrium manifold $\mathcal{Q}$ of a dynamical system $\Sigma$ on an analytic Riemannian manifold $(\mathcal{M},g)$ is referred to as almost globally asymptotically stable (AGAS) if it is stable and the flow $\Phi(t,x_0)$ of $\Sigma$ satisfies $\lim_{t\rightarrow\infty}d_g(\mathcal{Q},\Phi(t,x_0))=0$ for all $x_0\in\mathcal{M}\backslash\mathcal{N}$, where $\mathcal{N}\subset\mathcal{M}$ has measure zero on $\mathcal{M}$.
\end{definition}

It is not immediately clear that $\mathcal{G}$ being $\mathcal{M}$-synchronizing implies that $\mathcal{C}$ is an AGAS equilibrium of \eqref{eq:descent}. Since \eqref{eq:descent} is a gradient descent of $U$, it cannot converge to any maximum of $U$. Morover, any saddle point of $U$ is unstable. However, a set of saddle points may still have a region of attraction with positive measure, in which case $\mathcal{C}$ cannot be AGAS. In \cite{markdahl2017tac}, we show that any connected graph is $\Sn$-synchronizing for the potential function $U=\sum_{\{i,j\}\in\E}a_{ij}(1-\langle\ve[i]{x},\ve[j]{x}\rangle)$. Moreover, we prove that $\mathcal{C}$ is AGAS. However, obtaining such results is not trivial. We limit the scope of this paper to  characterizing $\St$-synchronizing graphs. Sufficient conditions for $\mathcal{C}$ to be AGAS will be established in future work. To give the reader an idea of the current state of the art in terms of characterizing $\mathcal{M}$-synchronizing graphs, we provide the following examples.

\begin{example}
On $\R^n$, for $U$ given by the potential function \eqref{eq:U} using the Euclidean vector norm, the system \eqref{eq:descent} becomes
\begin{align*}
\vd[i]{x}&=\sum_{j\in\Ni}a_{ij}(\ve[j]{x}-\ve[i]{x}),\,\forall\,i\in\V.
\end{align*}
The consensus manifold of this system is well-known to be globally asymptotically stable (which also implies AGAS), although that fact is usually not expressed by saying that any connected graph is $\R^n$-synchronizing.
\end{example}

\begin{example}
On $\GC$, the dynamics \eqref{eq:descent} can be expressed in polar coordinates $\vartheta_i\in\R$ for all $i\in\V$ as
\begin{align}\label{eq:kuramoto}
\dot{\vartheta}_i&=\sum_{j\in\Ni}a_{ij}\sin(\vartheta_j-\vartheta_i),\,\forall\,i\in\V.
\end{align}
For a complete graph $\mathcal{G}$, \eqref{eq:kuramoto} is equivalent to the Kuramoto model in the case of homogeneous agents, \ie identical oscillator frequencies. The cycle graph
\begin{align*}
\mathcal{G}_N=(\V,\{\{i,j\}\subset\V\,|\,|i-j|=1 \}\cup\{\{1,N\}\})
\end{align*}
is $\GC$-synchronizing for $N\leq4$ but not for $N\geq5$ \cite{sarlette2009geometry}. The problem of characterizing all  $\GC$-synchronizing graphs is open  \cite{sarlette2009geometry,dorfler2014synchronization}.
\end{example}

\begin{example}
On $\Sn$, the dynamics are given by
\begin{align*}
\vd[i]{x}&=\sum_{j\in\Ni}a_{ij}\ve[j]{x}-\Bigl\langle\sum_{j\in\Ni}a_{ij}\ve[j]{x},\ve[i]{x}\Bigr\rangle\ve[i]{x},\,\forall\,i\in\V.
\end{align*}
Any connected graph is $\Sn$-synchronizing for all $n\geq2$ \cite{markdahl2017tac}. 
\end{example}

\subsection{Problem Statement}

\noindent The aim of this paper is to identify instances of the Stiefel manifold that are $\St$-synchronizing. We know that all connected graphs are $\mathsf{St}(1,n)$-synchronizing for $n\geq3$ \cite{markdahl2017tac}. We also know that not all connected graphs are $\GC$-synchronizing \cite{sarlette2009geometry}, nor $\SOT$-synchronizing \cite{tron2012intrinsic}. Note that $\GC\simeq\mathsf{SO}(2)\simeq\mathsf{St}(1,2)$ and $\SOT\simeq\mathsf{St}(2,3)$. The question, to which we give a partial answer, is how the results of \cite{markdahl2017tac} extends to the cases of $2\leq p\leq n-2$. 

\section{Main Result}

\noindent Theorem \ref{th:main} is the main result of this paper. First we outline the proof. The details are provided in Section \ref{sec:critical} to \ref{sec:integer}.


\begin{theorem}\label{th:main}
Let the pair $(p,n)$ satisfy $p\leq\tfrac{2n}{3}-1$. All minimizers of the potential function $U$ given by \label{eq:VSt} with $a_{ij}=1$ for all $e\in\E$ belong to the consensus manifold; \ie all connected graphs are $\St$-synchronizing.
\end{theorem}

\begin{proof} Let $q:\St^N\times\ts[\St]{}^N{\,}\rightarrow\R$ denote the quadratic form obtained from the intrinsic Hessian of $U$ evaluated at a critical point. It can either be calculated from the Lagrange optimality conditions or by linearizing the dynamics at an equilibrium; we take the latter approach. The second-order necessary conditions for unconstrained optimization over Riemannian manifolds imply that $q$ is weakly positive at any minimum $\{\ma[i]{S}\}_{i=1}^N$ of $U$. Our goal is to exclude minimality of  all equilibria $(\ma[i]{S})_{i=1}^N\notin\mathcal{C}$ by finding $(\ma[i]{\Delta})_{i=1}^N\in\mathsf{X}{\,}_{i=1}^N\ts[\St]{i}$ such that  $q((\ma[i]{S})_{i=1}^N,(\ma[i]{\Delta})_{i=1}^N))<0$. It is clear that all elements of $\mathcal{C}$ are global minimizers since $U\geq0$ and $U|_{\mathcal{C}}=0$.

Based on our previous work \cite{markdahl2017tac}, we consider perturbations along the tangent space of $\mathcal{C}$. We do not need to determine an exact expression for the desired perturbation, it suffices to prove that it exists. Showing that $q$ is negative amounts to solving a nonconvex constrained optimization problem. This is done using the Lagrange optimality conditions. Here, we introduce the inequality $p\leq\tfrac{2n}{3}-1$ to fix a variable in the optimization problem. For this case, we can show that there is a perturbation such that $\mathcal{C}$ maximizes $q$ with an objective value of zero. Any other equilibrium configuration gives a strictly negative objective value. Throughout all these steps, we never utilize any particular property of the graph topology besides connectedness. Hence the final result applies to any connected graph. It follows that all connected graphs are $\St$-synchronizing.\end{proof}

\subsection{Critical Points}
\label{sec:critical}

\noindent Let $V:(\R^{n\times p})^N\rightarrow[0,\infty)$ be the smooth extension of $U$ obtained by relaxing the requirement $(\ma[i]{S})_{i=1}^N\in\St^N$ for all $i\in\V$. Then
\begin{align}
\md[i]{S}&=-\inabla_iU=-\Pi_i\nabla_i V=\Pi_i\sum_{j\in\Ni}\ma[j]{S}\nonumber\\
&=\ma[i]{S}\skews\Bigl(\mat[i]{S}\sum_{j\in\Ni}\ma[j]{S}\Bigr)+(\ma[n]{I}-\ma[i]{S}\mat[i]{S})\sum_{j\in\Ni}\ma[j]{S}.\label{eq:stateeq}
\end{align}

\begin{remark}
The system of homogeneous state-space oscillators on $\St$ given by
\begin{align}
\md[i]{S}=\ma{\Omega}\ma[i]{S}+\ma[i]{S}\ma{\Xi}-\inabla_i U,\label{eq:homo}
\end{align}
where $\ma{\Omega}\in\so$ and $\ma{\Xi}\in\mathsf{so}(p)$ can be reduced to \eqref{eq:stateeq} by a change of variables. To verify this, let $\ma{R}=\exp(-t\ma{\Omega})\in\SO$, $\ma{Q}=\exp(-t\ma{\Xi})\in\mathsf{SO}(p)$, form $\ve[i]{X}=\ma{R}\ma[i]{S}\ma{Q}\in\St$, and calculate
\begin{align*}
\md[i]{X}={}&-\ma{\Omega}\ma{R}\ma[i]{S}\ma{Q}+\ma{Q}\md[i]{S}\ma{R}-\ma{R}\ma[i]{S}\ma{Q}\ma{\Xi}\\
={}&\ma[i]{X}\skews\Bigl(\mat[i]{X}\sum_{j\in\Ni}\ma[j]{X}\Bigr)+(\ma[n]{I}-\ma[i]{X}\mat[i]{X})\sum_{j\in\Ni}\ma[j]{X},
\end{align*}
where we used that $[\ma{\Omega},\ma{R}]=\ma{0}$,  $[\ma{\Xi},\ma{Q}]=\ma{0}$. Put $(p,n)=(1,2)$ and let $\mathcal{G}$ be the complete graph to obtain the Kuramoto model of a system of homogeneous oscillators from \eqref{eq:homo}.
\end{remark}

The critical points of $U$ are the equilibria of \eqref{eq:stateeq}. At an equilibrium,
\begin{align*}
\ma{0}&=\ma[i]{S}\skews\Bigl(\mat[i]{S}\sum_{j\in\Ni}\ma[j]{S}\Bigr)+(\ma[n]{I}-\ma[i]{S}\mat[i]{S})\sum_{j\in\Ni}\ma[j]{S}.
\end{align*}
Since these two terms are orthogonal, we get
\begin{align}
\skews\mat[i]{S}\sum_{j\in\Ni}\ma[j]{S}=\ma{0},\,(\ma[n]{I}-\ma[i]{S}\mat[i]{S})\sum_{j\in\Ni}\ma[j]{S}=\ma{0}.\label{eq:eq}
\end{align}
Assume \eqref{eq:eq} holds. Denote $\ve[i]{V}=\sum_{j\in\Ni}\ma[j]{S}$. Since $\ma[i]{V}=\ma[i]{S}\mat[i]{S}\ma[i]{V}$, it follows that $\ma[i]{V}\in\im\ma[i]{S}$. Hence $\ve[i]{V}=\ma[i]{S}\ma[i]{P}$ for some $\ma[i]{P}\in\R^{p\times p}$. Moreover, since $\skews\mat[i]{S}\ma[i]{V}=\skews\ma[i]{P}=\ma{0}$, we find that $\ma[i]{P}$ is symmetric.

\subsection{The Intrinsic Hessian}

\noindent Let $W_{i,st}:\R^{N\times n\times p}\rightarrow\R$ be the smooth extension of $(\inabla_i U)_{st}=\langle\ve[s]{e},\inabla_i U\ve[t]{e}\rangle:(\St)^N\rightarrow\R$ obtained by relaxing the constraint $\ma[i]{S}\in\St$ to $\ma[i]{S}\in\R^{n\times p}$ for all $i$. Using the rules governing derivatives of inner products with respect to matrices \cite{petersen2008matrix}, introducing  $\ma[st]{E}=\ve[s]{e}\vet[t]{e}$, after a few calculations, we obtain
\begin{align*}	
\nabla_kW_{i,st}={}&\begin{cases}
-\Pi_i\ma[st]{E} & \textrm{ if } k\in\Ni,\\
\ma[st]{E}\skews\left(\mat[i]{S}\sum_{j\in\Ni}\ma[j]{S}\right)+&\\
\sum_{j\in\Ni}\ma[j]{S}\syms(\mat[i]{S}\ma[st]{E})+& \\	
\ma[st]{E}\sum_{j\in\Ni}\mat[j]{S}\ma[i]{S}&\textrm{ if } k=i,\\
\ma{0}& \textrm{ otherwise.}
\end{cases}
\end{align*}
Evaluate at an equilibrium, where $\sum_{j\in\Ni}\ma[j]{S}=\ma[i]{S}\ma[i]{P}$ and $\ma[i]{P}\in\R^{p\times p}$ is symmetric by Section \ref{sec:critical}, to find
\begin{align*}
\nabla_kW_{i,st}={}&\begin{cases}
-\Pi_i\ma[st]{E} & \textrm{ if } k\in\Ni,\\
\ma[i]{S}\ma[i]{P}\syms(\mat[i]{S}\ma[st]{E})+\ma[st]{E}\ma[i]{P} & \textrm{ if } k=i,\\
\ma{0}& \textrm{ otherwise.}
\end{cases}
\end{align*}

The intrinsic Hessian is a $(N\times n\times p)^2$-tensor consisting of $N^2np$ blocks $\ma[ki,st]{H}\in\R^{n\times p}$, which are obtained by projecting the extrinsic Hesssian on the tangent space of $\ma[k]{S}$
\begin{align*}
\ma[ki,st]{H}&=\inabla_{k}(\inabla_{i}U)_{st}=\Pi_k\nabla_{k}W_{i,st}\\
&=\Pi_k\nabla_{k}(\Pi_i\nabla_iV)_{st}.
\end{align*}

\subsection{The Quadratic Form}

\noindent Consider the quadratic form $q:\St^N\times\ts[\St]{}^N{\,}\rightarrow\R$ obtained from the intrinsic Hessian evaluated at an equilibrium for some perturbation $(\ma[i]{\Delta})_{i=1}^N\in\ts[\mathcal{C}]{}$,
\begin{align*}
q={}&\sum_{i=1}^N\sum_{k=1}^N\langle\ma[i]{\Delta},[\langle\ma[k]{\Delta},\inabla_k(\inabla_i U)_{st}\rangle]\rangle\\
={}&\sum_{i=1}^N\sum_{k=1}^N\langle\Pi_i\ma{\Delta},[\langle\Pi_k\ma{\Delta},\Pi_k\nabla_kW_{i,st}\rangle]\rangle,
\end{align*}
where $\ma{\Delta}\in\R^{n\times p}$. 


Note that 
$\langle\Pi_k\ma{X},\Pi_k\ma{Y}\rangle=\langle\Pi_k\ma{X},\ma{Y}\rangle$. The quadratic form is hence
\begin{align*}
q={}&\sum_{i=1}^N\sum_{k=1}^N\langle\Pi_i\ma{\Delta},[\langle\Pi_k\ma{\Delta},\nabla_kW_{i,st}\rangle]\rangle.
\end{align*}
Denote $p_{ki,st}=\langle\Pi_k\ma{\Delta},\nabla_kW_{i,st}\rangle$. Then
\begin{align*}
p_{ki,st}=\begin{cases}
\langle\Pi_k\ma{\Delta},-\Pi_i\ma[st]{E}\rangle\\
\langle\Pi_i\ma{\Delta},\ma[i]{S}\ma[i]{P}\syms(\mat[i]{S}\ma[st]{E})+\ma[st]{E}\ma[i]{P}\rangle\\
\ma{0}
\end{cases}
\end{align*}
for the cases of $k\in\Ni$, $k=i$, and $k\notin\Ni\cup\{i\}$ respectively. Denote $p_{ki}=[p_{st}]$ and calculate
\begin{align*}
p_{ki}&=\begin{cases}
-\Pi_i\Pi_k\ma{\Delta} & \textrm{ if } k\in\Ni,\\
\ma[i]{S}\syms\mat[i]{V}\Pi_i\ma{\Delta}+\Pi_i(\ma{\Delta})\ma[i]{P} & \textrm{ if } k=i,\\
\ma{0} & \textrm{ otherwise.}
\end{cases}
\end{align*}

To see this, consider each case separately. For $k\in\Ni$,
\begin{align*}
p_{ki,st}={}&\langle(\ma[n]{I}-\Pi_i+\Pi_i)\Pi_k\ma{\Delta},-\Pi_i\ma[st]{E}\rangle\\
={}&-\langle\Pi_i\Pi_k\ma{\Delta},\ma[st]{E}\rangle=-(\Pi_i\Pi_k\ma{\Delta})_{st},
\end{align*}
whereby $p_{ki}=-\Pi_i\Pi_k\ma{\Delta}$. For the case of $k=i$,
\begin{align*}
p_{ii,st}={}&\langle\Pi_i\ma{\Delta},\ma[i]{S}\ma[i]{P}\syms(\mat[i]{S}\ma[st]{E})+\ma[st]{E}\ma[i]{P}\rangle\\
={}&\tfrac12(\ma[i]{S}\mat[i]{V}\Pi_i\ma{\Delta})_{st}+\tfrac12(\ma[i]{S}(\Pi_i\ma{\Delta})\mtr\ma[i]{V})_{st}+\\
&(\Pi_i(\ma{\Delta})\ma[i]{P})_{st},
\end{align*}
whereby $p_{ii}=\ma[i]{S}\syms\mat[i]{V}\Pi_i\ma{\Delta}+\Pi_i(\ma{\Delta})\ma[i]{P}$.

This gives us the quadratic form
\begin{align*}
q={}&\sum_{i=1}^N\sum_{k=1}^N\langle\Pi_i\ma{\Delta},p_{ki}\rangle\\
={}&\sum_{e\in\mathcal{E}}\langle\Pi_i\ma{\Delta},p_{ki}\rangle+\langle\Pi_k\ma{\Delta},p_{ik}\rangle+\sum_{i\in\mathcal{V}}\langle\Pi_i\ma{\Delta},p_{ii}\rangle
\end{align*}
For ease of notation, let $q=2\sum_{e\in\E}q_{ik}+\sum_{i\in\V}q_i$, where
\begin{align*}
q_{ik}={}&\langle\Pi_i\ma{\Delta},p_{ki}\rangle=-\langle\Pi_i\ma{\Delta},\Pi_k\ma{\Delta}\rangle=q_{ki},\\
q_i={}&\langle\Pi_i\ma{\Delta},p_{ii}\rangle.
\end{align*}
%


Calculate
\begin{align*}
q_{ik}={}&-\langle\Pi_i\ma{\Delta},\Pi_k\ma{\Delta}\rangle\\
={}&\trace(-\mat{\Delta}\ma{\Delta}+\tfrac12\mat{\Delta}\ma[i]{S}\mat[i]{S}\ma{\Delta}+\tfrac12\mat[i]{S}\ma{\Delta}\mat[i]{S}\ma{\Delta}+\\
&\hphantom{\trace(}\tfrac12\mat{\Delta}\ma[k]{S}\mat[k]{S}\ma{\Delta}+\tfrac12\mat{\Delta}\ma[k]{S}\mat{\Delta}\ma[k]{S}-\\
&\hphantom{\trace(}\tfrac14\mat{\Delta}\ma[i]{S}\mat[i]{S}\ma[k]{S}\mat[k]{S}\ma{\Delta}-\tfrac14\mat{\Delta}\ma[i]{S}\mat[i]{S}\ma[k]{S}\mat{\Delta}\ma[k]{S}-\\
&\hphantom{\trace(}\tfrac14\mat[i]{S}\ma{\Delta}\mat[i]{S}\ma[k]{S}\mat[k]{S}\ma{\Delta}-\tfrac14\mat[i]{S}\ma{\Delta}\mat[i]{S}\ma[k]{S}\mat{\Delta}\ma[k]{S}).
\end{align*}
Use the identity $\trace\ma{ABCD}=\langle\vect\mat{A},(\mat{D}\otimes\ma{B})\vect\ma{C}\rangle$ \cite{graham1981kronecker} and the notation $\ve[1]{d}=\vect\ma{\Delta}$, $\ve[2]{d}=\vect\mat{\Delta}$ to write $q_{ik}=\langle\ve{d},\ma[ik]{Q}\ve{d}\rangle$, where $\ma[ik]{Q}$ is given in Table \ref{tab:matrix} and $\ve{d}=[\vet[1]{d}\,\vet[2]{d}]\mtr$.
\begin{table*}
\begin{align*}
\ma[ik]{Q}={}&\begin{bmatrix}
-\ma[np]{I}+\tfrac12\ma[p]{I}\otimes(\ma[i]{S}\mat[i]{S}+\ma[k]{S}\mat[k]{S})-\tfrac14\ma[p]{I}\otimes\ma[i]{S}\mat[i]{S}\ma[k]{S}\mat[k]{S} & \tfrac12\mat[k]{S}\otimes\ma[k]{S}-\tfrac14\mat[k]{S}\otimes\ma[i]{S}\mat[i]{S}\ma[k]{S}\\
\tfrac12\ma[i]{S}\otimes\mat[i]{S}-\tfrac14\ma[i]{S}\otimes\mat[i]{S}\ma[k]{S}\mat[k]{S} & -\tfrac14\ma[i]{S}\mat[k]{S}\otimes\mat[i]{S}\ma[k]{S}
\end{bmatrix}
\end{align*}
\caption{The matrix $\ma[ik]{Q}$.\label{tab:matrix}}
\end{table*}
%

Furthermore,
\begin{align*}
q_i={}&\langle\Pi_i\ma{\Delta},\ma[i]{S}\syms\mat[i]{V}\Pi_i\ma{\Delta}+\Pi_i(\ma{\Delta})\ma[i]{P}\rangle=\langle\ve{d},\ma[i]{Q}\ve{d}\rangle,
\end{align*}
where
\begin{align*}
\ma[i]{Q}={}&\begin{bmatrix}
\ma[i]{P}\otimes\ma[n]{I}-\tfrac34\ma[i]{P}\otimes\ma[i]{S}\mat[i]{S} & \ma{0}\\
-\tfrac12\ma[i]{S}\ma[i]{P}\otimes\mat[i]{S} & \tfrac14\ma[i]{S}\ma[i]{P}\mat[i]{S}\otimes\ma[p]{I}
\end{bmatrix}.
\end{align*}

There is a constant permutation matrix $\ma{K}\in\mathsf{O}(np)$ such that $\vect\mat{\Delta}=\ma{K}\vect\ma{\Delta}$ for all $\vect\ma{\Delta}\in\R^{np}$ \cite{graham1981kronecker}. Hence
\begin{align*}
\ve{d}=\begin{bmatrix}
\vect\ma{\Delta}\\
\vect\mat{\Delta}\\
\end{bmatrix}=\begin{bmatrix}
\ma[np]{I}\\
\ma{K}
\end{bmatrix}
\vect\ma{\Delta}=\begin{bmatrix}
	\ma[np]{I}\\
	\ma{K}
\end{bmatrix}
\ve[1]{d}.
\end{align*}
The quadratic form $q$ satisfies $q=\langle\ve[1]{d},\ma{M}\ve[1]{d}\rangle$, where
\begin{align*}
\ma{M}=\begin{bmatrix}
\ma[np]{I} & \!\!\!\mat{K} \end{bmatrix}\ma{Q}\begin{bmatrix}
\ma[np]{I}\\
\ma{K}
\end{bmatrix},\quad\ma{Q}=\sum_{i\in\V}\ma[i]{Q}+\sum_{k\in\Ni}\ma[ik]{Q}.
\end{align*}
%


%
%

We wish to show that $q$ assumes positive values for some $\ma{\Delta}\in\R^{n\times p}$ at all equilibria $(\ma[i]{S})_{i=1}^N\notin\mathcal{C}$. If $\trace\ma{M}$ is positive, then the symmetric part of $\ma{M}$ has at least one positive eigenvalue. Hence calculate
\begin{align*}
\trace\ma{M}&=\trace\begin{bmatrix}
\ma[np]{I} & \!\!\!\mat{K}
\end{bmatrix}\ma{Q}\begin{bmatrix}
\ma[np]{I}\\
\ma{K}
\end{bmatrix}=\trace\left(\ma{Q}\begin{bmatrix}
\ma[np]{I} & \mat{K}\\
\ma{K} & \ma[np]{I}
\end{bmatrix}\right)\\
&=\trace\left(\begin{bmatrix}\ma{A} & \ma{B}\\
\ma{C} & \ma{D}\end{bmatrix}
\begin{bmatrix}
\ma[np]{I} & \mat{K}\\
\ma{K} & \ma[np]{I}
\end{bmatrix}\right)\\
&=\trace(\ma{A}+\ma{B}\ma{K}+\ma{C}\mat{K}+\ma{D}).
\end{align*}
Omitting the details of calculations, it holds that
\begin{align*}
\trace\ma{A}={}&2\sum_{e\in\E}\left(n-\tfrac{3p}4\right)\left\langle\ma[k]{S},\ma[i]{S}\right\rangle-(n-p)p-\tfrac{p}4\|\mat[k]{S}\ma[i]{S}\|^2,\\
\trace\ma{D}={}&2\sum_{e\in\E}\tfrac{p}4\left\langle\ma[k]{S},\ma[i]{S}\right\rangle-\tfrac14\langle\ma[k]{S},\ma[i]{S}\rangle^2.
\end{align*}
To deal with terms involving $\ma{K}$, we utilize that $\ma{K}=\sum_{a=1}^n\sum_{b=1}^p\ma[ab]{E}\otimes\ma[ba]{E}$, where the elemental matrix $\ma[ab]{E}\in\R^{n\times p}$ is given by $\ma[ab]{E}=\ve[a]{e}\otimes\ve[b]{e}$ for all $a\in\{1,\ldots,n\}$, $b\in\{1,\ldots,p\}$ \cite{graham1981kronecker}. After some calculations we obtain
\begin{align*}
\trace\ma{B}\ma{K}={}&2\sum_{e\in\E}\tfrac{p}2-\tfrac14\|\mat[k]{S}\ma[i]{S}\|^2,\\
\trace\ma{C}\mat{K}\!\!={}&2\sum_{e\in\E}-\tfrac12\langle\ma[k]{S},\ma[i]{S}\rangle+\tfrac{p}2-\tfrac14\|\mat[i]{S}\ma[k]{S}\|^2.
\end{align*}

Adding up all four terms gives
\begin{align}
\tfrac12\trace\ma{M}={}&\sum_{e\in\E}\left(n-\tfrac{p+1}2\right)\left\langle\ma[k]{S},\ma[i]{S}\right\rangle-\tfrac{p+2}4\|\mat[k]{S}\ma[i]{S}\|^2-\nonumber\\
&\hphantom{\sum_{e\in\E}}\,\,\,\tfrac14\langle\ma[k]{S},\ma[i]{S}\rangle^2+(1-n+p)p.\label{eq:trM}
\end{align}
At a consensus we get $\trace\ma{M}|_{\mathcal{C}}=0$. This is expected since $U$ is constant over $\mathcal{C}$ and $\mathcal{C}$ is invariant under any perturbation that belongs to its tangent space. We also note that the result \eqref{eq:trM} is consistent with that of \cite{markdahl2017tac}. 

\subsection{Nonlinear Programming Problem}
\label{sec:nlp}

\noindent Having determined $q=\trace\ma{M}$ in \eqref{eq:trM}, it remains to show that $\trace\ma{M}$ is strictly negative for any configuration $(\ma[i]{S})_{i=1}^N\notin\mathcal{C}$. To that end, consider
\begin{align}
\tfrac12\trace\ma{M}\leq{}& |\E|\max_{\ma{X},\ma{Y}}f(\ma{X},\ma{Y}),\label{ineq:C}\\
f(\ma{X},\ma{Y})={}&\left(n-\tfrac{p+1}2\right)\left\langle\ma{X},\ma{Y}\right\rangle-\tfrac{p+2}4\|\mat{X}\ma{Y}\|^2-\nonumber\\
&\tfrac14\langle\ma{X},\ma{Y}\rangle^2+(1-n+p)p,\label{eq:f}
\end{align}
where $f:\St\times\St\rightarrow\R$. Hence, if we can show that the image of $f$ is negative for all $\ma{X}\neq\ma{Y}$, then we are done. Note that the inequality is sharp in the case of two agents and that $f(\ma{X},\ma{X})=0$ since this corresponds to consensus in a system of two agents.

Consider the nonlinear, non-convex optimization problem
\begin{align}
\label{eq:opt}
\begin{split}
&\max f(\ma{X},\ma{Y})\,\,\textrm{s.t.}\,\ma{X},\ma{Y}\in\St,
\end{split}
\end{align}
where $f(\ma{X},\ma{Y})$ is given by \eqref{eq:f}. It follows from \eqref{ineq:C} that \eqref{eq:opt} is a relaxation of the problem $\max\trace\ma{M}$ such that $(\ma[i]{S})_{i=1}^N\in\St^N$ and the equations \eqref{eq:eq} hold.

Problem \eqref{eq:opt} can be solved through use of the Lagrange conditions for optimality. To that end, introduce the functions $g_{st}(\ma{X})=\langle\ma{X}\ve[s]{e},\ma{X}\ve[t]{e}\rangle-\delta_{st}$, where $\delta_{\cdot,\cdot}$ denotes the Kronecker delta. The constraints in \eqref{eq:opt} can be summarized as $g_{st}(\ma{X})=0,\,g_{st}(\ma{Y})=0$ for all $s,t\in\{1,\ldots,p\}$. Form the Lagrangian
\begin{align*}
\mathcal{L}=f(\ma{X},\ma{Y})+\sum_{s,t}\lambda_{st}g_{st}(\ma{X})+\xi_{st}g_{st}(\ma{Y}),
\end{align*}
where $\lambda_{st}$, $\xi_{st}$, $s,t\in\{1,\ldots,p\}$, are Lagrange multipliers.

Partial derivatives of the objective function are given by
\begin{align*}
\tfrac{\partial}{\partial\ma{X}}\mathcal{L}&=\left(n-\tfrac{p+1}{2}\right)\ma{Y}-\tfrac{p+2}2\ma{Y}\mat{Y}\ma{X }-\tfrac12\langle\ma{X},\ma{Y}\rangle\ma{Y}+\ma{X}\ma{\Lambda},\\
\tfrac{\partial}{\partial\ma{Y}}\mathcal{L}&=\left(n-\tfrac{p+1}2\right)\ma{X}-\tfrac{p+2}2\ma{X}\mat{X}\ma{Y}-\tfrac12\langle\ma{X},\ma{Y}\rangle\ma{X}+\ma{X}\ma{\Xi},
\end{align*}
where
\begin{align*}
\ma{\Lambda}&=\sum_{s,t}\lambda_{st}(\ve[s]{e}\vet[t]{e}+\ve[t]{e}\vet[s]{e})=[\lambda_{st}]+[\lambda_{ts}]=[\lambda_{st}+\lambda_{ts}],\\
\ma{\Xi}&=\sum_{s,t}\xi_{st}(\ve[s]{e}\vet[t]{e}+\ve[t]{e}\vet[s]{e})=([\xi_{st}]+[\xi_{ts}])=[\xi_{st}+\xi_{ts}],
\end{align*}
are symmetric. The critical points of $\mathcal{L}$ satisfy
\begin{align}
\left(n-\tfrac{p+1}{2}-\tfrac12\langle\ma{X},\ma{Y}\rangle\right)\ma{Y}-\tfrac{p+2}2\ma{Y}\mat{Y}\ma{X }+\ma{X}\ma{\Lambda}&=\ma{0},\label{eq:X}\\
\left(n-\tfrac{p+1}2-\tfrac12\langle\ma{X},\ma{Y}\rangle\right)\ma{X}-\tfrac{p+2}2\ma{X}\mat{X}\ma{Y}+\ma{Y}\ma{\Xi}&=\ma{0},\label{eq:Y}
\end{align}
$[g_{st}(\ma{X})]=\mat{X}\ma{X}-\ma{I}=\ma{0}$, and $[g_{st}(\ma{Y})]=\mat{Y}\ma{Y}-\ma{I}=\ma{0}$.

Solve this system for $\ma{\Lambda}$ and $\ma{\Xi}$. Introduce $\ma{Z}=\mat{X}\ma{Y}$. Multiply \eqref{eq:X} and \eqref{eq:Y} from the left by $\mat{X}$ and $\mat{Y}$ respectively, to find
\begin{align*}
\ma{\Lambda}=-\left(n-\tfrac{p+1}{2}-\tfrac12\trace\ma{Z}\right)\ma{Z}+\tfrac{p+2}2\ma{Z}\mat{Z },\,\,
\ma{\Xi}&=\mat{\Lambda}.
\end{align*}
Since $\ma{\Lambda}$ and $\ma{\Xi}$ are symmetric, it holds that
\begin{align*}
\skews\ma{\Lambda}=-\left(n-\tfrac{p+1}{2}-\tfrac12\trace\ma{Z}\right)\skews\ma{Z}=-\skews\ma{\Xi}=\ma{0}.
\end{align*}
Note that $n-\tfrac{p+1}{2}-\tfrac12\trace\ma{Z}\geq n-p-\tfrac12>0$ for $p<n$ since $\trace\ma{Z}\leq p$, wherefore $\skews\ma{Z}=\ma{0}$ (the case of $p=n$ is  excluded from consideration since $\mathsf{St}(n,n)\simeq\mathsf{O}(n)$ is not path connected). This implies that $\ma{\Xi}=\ma{\Lambda}$.

Substitute
\begin{align*}
\ma{\Lambda}=-\left(n-\tfrac{p+1}{2}-\tfrac12\trace(\ma{Z})\right)\ma{Z}+\tfrac{p+2}2\ma{Z }^2
\end{align*}
into the equation $\mat{Y}\nabla_{\ma{X}}\mathcal{L}=\ma{0}$ to find
\begin{align*}
\mat{Y}\nabla_{\ma{X}}\mathcal{L}={}&\left(n-\tfrac{p+1}{2}-\tfrac12\trace\ma{Z}\right)\ma{I}-\tfrac{p+2}2\ma{Z}+\\
&\ma{Z}\left(-\left(n-\tfrac{p+1}{2}-\tfrac12\trace(\ma{Z})\right)\ma{Z}+\tfrac{p+2}2\ma{Z }^2\right)=\ma{0}
\end{align*}
By simplifying, we obtain
\begin{align*}
\mat{Y}\nabla_{\ma{X}}\mathcal{L}&=\left(\left(n-\tfrac{p+1}{2}-\tfrac12\trace\ma{Z}\right)-\tfrac{p+2}2\ma{Z}\right)(\ma{I}+\ma{Z})(\ma{I}-\ma{Z}).
\end{align*}
It follows that
\begin{align}
p(z)&=\left(\tfrac{2n-p-1-\trace\ma{Z}}{p+2}-z\right)(1+z)(1-z)\label{eq:minimal}
\end{align}
is the minimal polynomial of $\ma{Z}$, up to the exclusion of any factors corresponding to non-singular matrices. 

\subsection{Integer Programming Problem}
\label{sec:integer}
\noindent Reformulate the nonlinear programming problem \eqref{eq:opt} as an integer programming problem in terms of the algebraic multiplicities of the eigenvalues of $\ma{Z}$. By \eqref{eq:minimal}, the spectrum of $\ma{Z}$ satisfies  $\sigma\subset\{-1,\lambda_*,1\}$, where
\begin{align}
\lambda_*=\tfrac{2n-p-1-\trace\ma{Z}}{p+2}.\label{eq:lambda}
\end{align} 
Let $m_-,m_*$, and $m_+\in\{0,\ldots,p\}$ denote the algebraic multiplicities of the  eigenvalues $-1,\lambda_*$, and $1$ respectively. First note that $\lambda_*\leq\|\ma{Z}\|_2\leq\|\ma{X}\|_2\|\ma{Y}\|_2=1$, wherefore $m_*=0$ if $\lambda_*>1$. We know an expression for $\lambda_*$ in terms of $\trace\ma{Z}$. We also know that $\trace\ma{Z}=-m_-+\lambda^*m_*+m_+$ and $p=m_-+m_*+m_+$. Solving for $\lambda_*$ in terms of $m_-$ and $m_*$, we find 
\begin{align*}
\lambda_*=\tfrac{2n-2p-1+2m_-+m_*}{p+2+m_*}.
\end{align*}
From the inequality $\lambda^*\leq1$ we obtain $m_-\leq\tfrac32(p+1)-n$. This contradicts $m_-\geq0$ if $p<\tfrac{2n}3-1$, in which case $m_*=0$. For the case of $p=\tfrac{2n}3-1$ we learn that either $m_*=0$ or $m_-=0$.  The case of $p>\tfrac{2n}3-1$ is less informative.

We consider two cases: that of $m_*=0$ for general $p$ and that of $m_-=0$ for $p=\tfrac{2n}3-1$. These two cases exhaust all instances of $\St$ for which the pair $(p,n)$ satisfies $p\leq\tfrac{2n}3-1$. Start with $m_*=0$. Then 
\begin{align*}
\trace\ma{Z}&=-m_-+m_+=2m_+-p,\\
\trace\ma{Z}^2&=m_-+m_+=p,\\
(\trace\ma{Z})^2&=4m_+^2-4pm_++p^2.
\end{align*}
Recast the nonlinear program \eqref{eq:opt} as the equivalent quadratic integer program
\begin{align}
\label{eq:int}
\max\,(2n+1)m_+-m_+^2
\,\,\,\textrm{s.t. }\,m_+\in\{0,\ldots,p\}
\end{align}
where we used $\ma{Z}=\mat{X}\ma{Y}$, $\langle\ma{X},\ma{Y}\rangle=\trace\ma{Z}$, $\|\ma{Z}\|^2=\trace\ma{Z}^2$, and removed the constant terms. 

Denote the objective function of Problem \eqref{eq:int} by $h:\N\rightarrow\R$. Note that $h(m_++1)=h(m_+)+2(n-m_+)$. Since $m_+\leq p\leq n$, the optimization problem is solved by maximizing $m_+$, \ie $m_+=p$. This corresponds to $\ma{Z}=\ma{I}$ and the optimal value of zero in \eqref{eq:opt}. Any suboptimal solution gives a strictly lower objective value. 

It remains to consider the case of $p=\tfrac{2n}{3}-1$, $m_-=0$. Calculate
\begin{align*}
\trace\ma{Z}&=\lambda_*m_*+m_+=p+(\lambda_*-1)m_*\\
\trace\ma{Z}^2&=p+(\lambda_*+1)(\lambda_*-1)m_*\\
(\trace\ma{Z})^2&=p^2+2(\lambda_*-1)pm_*+(\lambda_*-1)^2m_*^2.
\end{align*}
Recast the nonlinear program \eqref{eq:opt} as the equivalent mixed integer nonlinear program
\begin{align}
\label{eq:int2}
\begin{split}
&\max\,-(\tfrac{n}6+\tfrac14+\tfrac14m_*)(1-\lambda_*)^2m_*,\\
&\,\,\textrm{s.t. }\,\,\,\,\lambda_*\in[0,1],\,\,m_*\in\{0,\ldots,p\},
\end{split}
\end{align}
where the constant term has been removed. It is clear that the optimal solution has $\lambda_*=1$ or $m_*=0$. It follows that $\sigma(\ma{Z})=\{1\}$, \ie $\ma{Z}=\ma{I}$.

\section{Conclusions}

\noindent This paper extends the results of \cite{markdahl2017tac} concerning almost global stability of a state space oscillator on $\Sn$ for $n\geq2$ to the case of $\St$. We characterize all connected graphs as $\St$-synchronizing provided that the pair $(p,n)$ satisfies $p\leq\tfrac{2n}{3}-1$. This inequality is sharp with respect to known results: some connected graphs are not $\GC$-synchronizing since $\GC\simeq\mathsf{St}(1,2)$ and $1\nleq\tfrac13$ \cite{sarlette2009geometry}; all connected graphs are $\mathcal{S}^{n-1}$-synchronizing since $\mathcal{S}^{n-1}\simeq\mathsf{St}(1,n)$ and $1\leq\tfrac{2n}3-1$ for $n\geq2$ \cite{markdahl2017tac}; some connected graphs are not $\SOT$-synchronizing since $\SOT\simeq\mathsf{St}(2,3)$ and $2\nleq1$ \cite{tron2012intrinsic}. Still, these results are likely to be conservative for $n\geq4$ due to the derivations involving a number of inequalities. 

\bibliographystyle{unsrt}
\bibliography{cdc2018bib}

\end{document}